\documentclass{article}[12pt]
\date{}

\usepackage{graphicx,tikz,color}
\usepackage{amsmath,amsfonts,amssymb,amsthm,latexsym}
\usepackage{enumerate}

\newtheorem{theorem}{Theorem}[section]

\newtheorem{conjecture}{Conjecture}
\newtheorem{corollary}[theorem]{Corollary}
\newtheorem{remark}[theorem]{Remark}
\newtheorem{lemma}[theorem]{Lemma}
\newtheorem{claim}[theorem]{Claim}

\newtheorem{problem}[theorem]{Problem}

\newenvironment{claimproof}[1]{{\it\noindent{Proof.}}\space#1}{\footnotesize \hfill \ensuremath{(\square)} \medskip}

\tikzstyle{vertex}=[circle, draw, inner sep=0pt, minimum size=6pt]
\newcommand{\es}{es_{\chi'}}

\begin{document}

\title{On the chromatic edge stability index of graphs}

\author{Saieed Akbari$^{a}$\thanks{Email: \texttt{s\textunderscore akbari@sharif.edu}}\and Arash Beikmohammadi$^{b}$\thanks{Email: \texttt{arash.beikmohammadi@gmail.com}}\and Bo\v stjan Bre\v sar$^{c,d}$\thanks{Email: \texttt{bostjan.bresar@um.si}}\and Tanja Dravec$^{c,d}$\thanks{Email: \texttt{tanja.dravec@um.si}}\and Mohammad Mahdi Habibollahi$^{b}$\thanks{Email: \texttt {m.habiballahi@gmail.com}}\and Nazanin Movarraei$^{e}$\thanks{Email: \texttt{nazanin.movarraei@gmail.com}}\smallskip }

\maketitle

\begin{center}
$^a$ Department of Mathematical Sciences, Sharif University of Technology, Tehran, Iran\\
\medskip

$^b$ Department of Computer Engineering, Sharif University of Technology, Tehran, Iran\\
\medskip

$^c$ Faculty of Natural Sciences and Mathematics, University of Maribor, Slovenia\\
\medskip

$^d$ Institute of Mathematics, Physics and Mechanics, Ljubljana, Slovenia\\
\medskip

$^e$Department of Mathematics, Yazd University, Yazd, Iran\\ 
\end{center}

\medskip

\begin{abstract}
Given a non-trivial graph $G$, the minimum cardinality of a set of edges $F$ in $G$ such that $\chi'(G \setminus F)<\chi'(G)$ is called the chromatic edge stability index of $G$, denoted by $\es(G)$, and such a (smallest) set $F$ is called a (minimum) mitigating set. While $1\le \es(G)\le \lfloor n/2\rfloor$ holds for any graph $G$, we investigate the graphs with extremal and near-extremal values of $\es(G)$. The graphs $G$ with $\es(G)=\lfloor n/2\rfloor$ are classified, and the graphs $G$ with $\es(G)=\lfloor n/2\rfloor-1$ and $\chi'(G)=\Delta(G)+1$ are characterized. We establish that the odd cycles and $K_2$ are exactly the regular connected graphs with the chromatic edge stability index $1$; on the other hand, we prove that it is NP-hard to verify whether a graph $G$ has $\es(G)=1$. We also prove that every minimum mitigating set of an $r$-regular graph $G$, where $r\ne 4$, with $\es(G)=2$ is a matching. Furthermore, we propose a conjecture that for every graph $G$ there exists a minimum mitigating set, which is a matching, and prove that the conjecture holds for graphs $G$ with $\es(G)\in\{1,2,\lfloor n/2\rfloor-1,\lfloor n/2\rfloor\}$, and for bipartite graphs. 

\end{abstract}


\noindent
{\bf AMS Subj.\ Class.\ (2010)}: 05C15, 05C70.

\section{Introduction}
Throughout this paper all graphs are finite and simple, that is, with no loops and multiple edges, and moreover, with at least one edge.
Given a graph $G$, a function $c:E(G) \to \{c_1,\ldots , c_k\}$ with $c(e) \neq c(f)$ for any two adjacent edges $e$ and $f$ is a  {\emph{proper $k$-edge coloring}} of $G$. The minimum $k$ for which $G$ admits a proper $k$-edge coloring is the {\emph{chromatic index}} of $G$, and denoted by $\chi'(G)$.
We let $[k]=\{1,\ldots,k\}$.
For any $i \in [\chi'(G)]$, let $C_i$ denote the set of all edges of $G$ that are colored by $c_i$ in the proper edge coloring $c$. 
For any $v \in V(G)$, let $c(v)$ denote the set of colors appearing in $v$.
The \emph{open neighborhood} of a vertex $v$ in $G$ is the set of neighbors of $v$, denoted by $N_G(v)$, whereas the \emph{closed neighborhood} of $v$ is $N_G[v] = N_G(v) \cup \{v\}$. The \emph{degree} of a vertex $v$ in $G$ is denoted by $d_G(v) = |N_G(v)|$. A $k$-\emph{regular graph} is a graph in which every vertex has degree~$k$. A graph $G$ is \emph{regular} if it is $k$-regular for some integer $k \ge 0$.
The subgraph of $G$ induced by $A \subseteq V(G)$ will be denoted by $G \left[ A \right]$. The {\it{Core(G)}} is the subgraph of $G$ induced by all vertices of maximum degree $\Delta(G)$. In this paper $\overline G$ denotes the {\em complement} of $G$, that is, $V(\overline G)=V(G)$ and $E(\overline G)=\overline{E(G)}$. 
The complete graph of order $n$ is denoted by $K_n$. The complete bipartite graph with part sizes $m$ and $n$ is denoted by $K_{m,n}$.

One of the most celebrated results in graph theory due to Vizing~\cite{viz-1964} states that the chromatic index of an arbitrary simple graph lies between the maximum degree $\Delta(G)$ and $\Delta(G)+1$. Graphs with $\chi'(G)=\Delta(G)$ are said to be of {\em Class $1$}, while graphs with $\chi'(G)=\Delta(G)+1$ are said to be of {\em Class $2$}. Holyer proved that determining whether a graph is of Class 1 (or 2) is in general an NP-complete problem~\cite{hol-1981}, which in part explains the large number of investigations of the chromatic index and related properties; see a recent survey on edge colorings~\cite{ccjst-2019}.

Vizing in a follow-up~\cite{viz-1965} investigated the effect of edge removal in a graph of Class 2, and called such a graph {\em critical} if after removing any edge the chromatic index drops. It is easy to see that the only critical Class 1 graphs are the stars. It is also easy to see that every Class 2 graph $G$ contains a subgraph $H$ such that $H$ is a critical Class 2 graph and $\Delta(H)=\Delta(G)$. Vizing later proposed two conjectures for critical Class 2 graphs, notably, in~\cite{viz-1965b} he conjectured that every such graph admits a $2$-factor, and in~\cite{viz-1968} that the independence number of a critical graph is at most half of its order. In spite of many attempts and partial results, the conjectures are still unresolved; see a recent study~\cite{ks-2019} and the references therein. 

In this paper, we propose a different perspective on the study of the effect of edge removal with respect to the chromatic index. A similar study with respect to the chromatic number was initiated by Staton~\cite{stat-1980} in 1980, and received a considerable attention in recent years~\cite{akmn-2019+, bkm-2020,kmn-2018}. The question considered in these papers is, how many edges need to be removed from a graph so that its chromatic number drops. To the best of our knowledge, an analogous question for the chromatic index has not yet been studied. (The only exception is a very recent paper~\cite{km-2020+} in which the concept was mentioned in a more general context, and some initial results were obtained.)
We find this surprising and want to initiate the investigation of this problem, which we formally define as follows.

The {\emph{chromatic edge stability index}}, $es_{\chi'}(G)$, of a graph $G$ is the minimum number of edges of $G$ such that their deletion results in a graph $H$ with $\chi'(H)=\chi'(G)-1$. For a set $F \subseteq E(G)$ we denote by $G \setminus F$ the graph $G$ without the edges of the set $F$. If $F=\{e\}$, where $e\in E(G)$, we simply write $G \setminus e$ instead of $G \setminus \{e\}$.  
A {\emph{mitigating set}} is a set of edges $F$ in $G$ such that $\chi'(G \setminus F) < \chi'(G)$.

A {\em matching} $M$ is a set of edges in $G$ such that every two distinct edges $e$ and $f$ in $M$ are not adjacent. We say that a matching $M$ with $|M|=k$ is a {\em $k$-matching}. A {\emph{perfect matching} of $G$ is a matching $M$ of $G$ that covers all vertices of $G$.} The cardinality of a maximum matching of a graph $G$ is the {\emph{matching number}} of $G$ and denoted by $\alpha'(G)$.

Let $c:E(G) \to \{c_1,\ldots , c_{\chi'(G)}\}$ be a proper $\chi'(G)$-edge coloring of $G$. Suppose that 
$C$ is a color class. Then, obviously $\chi'(G \setminus C) < \chi'(G)$, and since $C$ is a matching, we infer $|C|\le \left\lfloor \frac{|V(G)|}{2} \right\rfloor$. This yields the upper bound in the following basic result (the lower bound is trivial).

\begin{lemma}\label{l:bound}
If $G$ is a graph of order $n$, then $1\le \es(G) \leq \lfloor \frac{n}{2} \rfloor$. 
\end{lemma}  

The following two questions will be studied in this paper. What are the graphs that achieve the extremal or near extremal values of the chromatic edge stability index in view of the inequalities in Lemma~\ref{l:bound}? The second question is concerned with the structure of minimum mitigating sets. Notably, it is not clear if there is always such a proper edge coloring of $G$ for which there is a minimum mitigating set obtained in this way (by deletion of edges of a smallest color class). We suspect the answer is affirmative, and propose this as the following conjecture. 

\begin{conjecture}\label{conj1}
For every graph $G$ there exists a minimum mitigating set which is a matching. 
\end{conjecture}

In Section~\ref{sec0} we establish the notation, present several preliminary results from the literature, and study the chromatic edge stability index in some simple graph families.  In Section~\ref{sec:es1}, we consider graphs with the extreme values ($1$ and $\lfloor \frac{n}{2} \rfloor$, respectively) of the chromatic edge stability index. We prove that a graph $G$ attains the upper bound if and only if $G$ is either a complete graph of odd order or a Class 1 regular graph of even order or a Class 1 graph of odd order with $\frac{(n-1)\Delta(G)}{2}$ edges. We also characterize the connected regular graphs $G$ with $\es(G)=1$. We prove that the problems of determining whether $\es(G)=\lfloor \frac{n}{2} \rfloor$ or $\es(G)=1$ are NP-hard. Then, in Section~\ref{sec:2}, graphs $G$ with $\es(G)=2$ are investigated, and we prove that every such graph has a minimum mitigating set consisting of two non-adjacent edges. In addition, we prove that in $r$-regular graphs $G$ with $\es(G)=2$, where $r\ne 4$, the only minimum mitigating sets are matchings.
In Section~\ref{sec:5}, Conjecture~\ref{conj1} is confirmed for graphs $G$ with 
$es_{\chi'}(G)= \lfloor \frac{n}{2} \rfloor-1$. Furthermore, we characterize graphs $G$ of Class 2 for which $\es(G)=\lfloor \frac{n}{2}\rfloor -1$.
In Section~\ref{sec:bip} we consider bipartite graphs and prove that Conjecture~\ref{conj1} holds for any bipartite graph. In the final section, we give concluding remarks and propose some open problems.

\section{Preliminary results}\label{sec0}

In this section, we establish the notation, and present the chromatic edge stability index in some simple families of graphs, such as cycles, paths, complete graphs and complete bipartite graphs. We also present several useful preliminary results from the literature concerning edge colorings.

The following observation is clear and was used also in the arguments for Lemma~\ref{l:bound}. 
\begin{remark}\label{r:C1}
If $C_1$ is a smallest color class of a proper $\chi'(G)$-edge coloring $c$ of $G$, then $\es(G) \leq |C_1|$.
\end{remark}

Consider a proper $\chi'(G)$-edge coloring $c$ of $G$ and let $c_{1}$ and $c_{2}$ be two distinct colors of $c$. A  {\em $(c_{1},c_{2})$-path} in $G$ is a maximal path of length at least one whose edges have colors $c_{1}$ and $c_{2}$, which alternate along the path. Furthermore, exactly one of the colors of $c_1$ and $c_2$ appears in the first vertex of this path (note that this maximal path does not terminate with the starting vertex). If $u\in V(G)$, $c_{1} \in c(u)$ and $c_{2} \notin c(u)$, then $P_{u}(c_{1},c_{2})$ denotes $(c_{1},c_{2})$-path, starting at $u$. 

\begin{remark}\label{r:path}
Let $c$ be a proper $\chi'(G)$-edge coloring of $G$ and suppose that there exists a path $P_{u}(c_{1},c_{2})$. If we switch the colors $c_{1}$ and $c_{2}$ in $P_{u}(c_{1},c_{2})$ without any change of the colors of the rest of the edges of $G$, this is also a proper $\chi'(G)$-edge coloring of $G$.
\end{remark}

Let us recall Vizing's fundamental theorem on the chromatic index of a simple graph.

\begin{theorem}{\rm \cite{viz-1964}}
\label{thm:vising}
If $G$ is a graph, then $\Delta(G)\leq\chi'(G)\leq\Delta(G)+1$.
\end{theorem}

By the above result, the family of all graphs can be partitioned into two classes: a graph $G$ is of {\em Class 1} if $\chi'(G)=\Delta(G)$, and otherwise, $G$ is of {\em Class 2}. It is well known that bipartite graphs are of Class 1, due to K\H{o}nig's theorem~\cite{kon-1916} from 1916.

\begin{theorem}{\rm \cite{kon-1916}}
\label{thm:konig}
If $G$ is a bipartite graph, then $\chi'(G)=\Delta(G)$.
\end{theorem}

\begin{lemma} {\rm \cite[p.96]{bondy-1982}}
\label{Bondy}
Let $G$ be a graph of odd order $n$ and $|E(G)|> \frac {(n-1) \Delta(G)}{2}$, then $G$ is of Class $2$.
\end{lemma}

Concerning the computational complexity of the chromatic edge stability index problem, we will make use of the following result due to Holyer. 

\begin{theorem}{\rm \cite{hol-1981}}
\label{thm:holyer}
It is NP-complete to determine whether the chromatic index of a $3$-regular graph is $3$ or $4$.
\end{theorem}

The next lemma is generally known as {\em Vizing's Adjacency Lemma}.

\begin{lemma} {\rm \cite{viz-1965}}
\label{VizingAdj}
Let $G$ be a critical (simple) graph and let $e=xy \in E(G)$. Then $x$ is adjacent to at least $\Delta(G)+1-d_G(y)$ vertices, distinct from $y$, having maximum degree.
\end{lemma}

In addition, the following result due to Fournier~\cite{four-1973} about the core of a graph will be used several times in the paper.

\begin{lemma} {\rm \cite{four-1973}}
\label{Fournier}
Let $G$ be a graph. If $Core(G)$ is a forest, then $G$ is of Class $1$.
\end{lemma}

The previous lemma was first proved by Fournier~\cite{four-1973}. This lemma is an immediate consequence of Vizing's Adjacency Lemma. To see this, by contradiction assume that $G$ is of Class 2. Clearly, $\Delta(G) \geq 2$. Let $G'$ be a critical subgraph of $G$ such that $\Delta(G')=\Delta(G)$ and $G'$ is of Class 2. Since $Core(G')$ is a subgraph of $Core(G)$, so $Core(G')$ is a forest. Let $u \in V(Core(G'))$ and $d_{Core(G')}(u) \leq1$. Since $\Delta(G') \geq 2$, there exists $w \in V(G') \setminus V(Core(G'))$ such that $uw \in E(G')$. By Lemma~\ref{VizingAdj}, $u$ is adjacent to at least $\Delta(G') - (\Delta(G')-1) + 1 = 2$ vertices of $Core(G')$, a contradiction.

In addition, the following result due to Akbari et al.~\cite{akbari-2011} about the core of a graph will be used several times in the paper.
A {\emph{unicyclic graph}} is a connected graph with exactly one cycle.

\begin{lemma} {\rm \cite{akbari-2011}}
\label{Akbari}
Let $G$ be a connected graph. If every connected component of $Core(G)$ is a unicyclic graph or a tree, and $Core(G)$ is not a disjoint union of cycles, then $G$ is of Class $1$.
\end{lemma}

Let $G$ be a graph. A {\emph{balanced edge coloring}} is a proper $\chi'(G)$-edge coloring of $G$ such that $||C_i|-|C_j||\leq 1$, for $1 \leq i,j \leq \chi'(G)$.
Balister et al.~\cite{bal-2002} proved that somewhat surprisingly every graph has such an edge coloring. 

\begin{lemma} {\rm \cite{bal-2002}}
\label{BalancedColoring}
Every graph $G$ has a balanced edge coloring.
\end{lemma}

In the rest of this section we present the chromatic edge stability index for some families of graphs. Since $\chi'(P_n)=2$ (for $n \geq 3$) and each graph $G$ with $\Delta(G)=2$ has $\chi'(G) \geq 2$, $\es(P_n)=\lfloor \frac{n-1}{2} \rfloor$ (for $n \geq 3$). For cycles, the chromatic edge stability index depends on the parity of $n$. If $n$ is odd, then $\es(C_n)=1$. If $n$ is even, then $\es(C_n)=\frac{n}{2}$, as $\chi'(C_n \setminus S) \leq 1$ if and only if $C_n \setminus S$ does not contain a vertex of degree 2.

For complete bipartite graphs, if $n \geq m$, then $\chi'(K_{m,n})=\Delta(K_{m,n})=n$. Since any subgraph of a bipartite graph is bipartite and of Class 1 (by Theorem~\ref{thm:konig}), $\chi'(K_{m,n} \setminus S) \leq n-1$ only if $\Delta(K_{m,n} \setminus S) \leq n-1$, for $S \subseteq E(G)$. Thus $|S| \geq m$. Since for any maximum matching $M$ of $K_{m,n}$, $\chi'(K_{m,n} \setminus M)=n-1$, $\es(K_{m,n}) \leq m$. Combining both inequalities, we get $\es(K_{m,n})=m$.

\begin{lemma}\label{l:complete}
If $G$ is a complete graph of order $n$, then $\es(G)=\lfloor \frac{n}{2} \rfloor.$
\end{lemma} 
\begin{proof}
First, let $n$ be even. Hence $\chi'(K_n)=n-1$, by Theorem~\ref{thm:vising} and {\rm \cite[p.274]{west-2001}}. Since each color class contains at most $\frac{n}{2}$ edges, we can color at most $\frac{n(n-2)}{2}$ edges with $n-2$ colors. Therefore, if $\chi'(K_n \setminus S)=n-2$, then $|S| \geq {n \choose 2}-\frac{n(n-2)}{2}=\frac{n}{2}$. Together with Lemma~\ref{l:bound} we get $\es(K_n)=\frac{n}{2}.$

Now, let $n$ be odd. Hence $\chi'(K_n)=n$, by Lemmas~\ref{thm:vising},~\ref{Bondy}. Since each color class contains at most $\frac{n-1}{2}$ edges, we can also color at most $\frac{(n-1)(n-1)}{2}$ edges with $n-1$ colors. Therefore, if $\chi'(K_n \setminus S)=n-1$, then $|S| \geq {n \choose 2}-\frac{(n-1)(n-1)}{2}=\frac{n-1}{2}= \lfloor \frac{n}{2} \rfloor$. Together with Lemma~\ref{l:bound} we get $\es(K_n)=\lfloor \frac{n}{2} \rfloor.$
\end{proof}

\section{Graphs with extreme values of $\es(G)$}\label{sec:es1}
In this section, we first consider the graphs with $\es(G)=\lfloor \frac{n}{2} \rfloor$ and then continue with investigating the graphs with $\es(G)=1$. The results yield that Conjecture~\ref{conj1} holds for both classes. 

Our first result shows that complete graphs $K_{n}$ (with odd $n$) are the only graphs of Class 2 with $\es(G)= \lfloor \frac{n}{2} \rfloor$.
\begin{theorem}\label{p:extremeClass2}
If $G$ is a Class $2$ graph of order $n$, then $\es(G)=\lfloor \frac{n}{2}\rfloor$ if and only if $n$ is odd and $G$ is isomorphic to $K_n$.
\end{theorem}
\begin{proof}
Let $G$ be a Class 2 graph of order $n$. 
We first prove that $\es(G) < \frac{n}{2}$, if $n$ is even. Therefore, let $G$ be a Class 2 graph of even order. By contradiction suppose that $\es(G) = \frac{n}{2}$. Then by Remark~\ref{r:C1} every color class of any proper $\chi'(G)$-edge coloring contains exactly $\frac{n}{2}$ edges. Thus $|E(G)|=(\Delta(G)+1)\frac{n}{2}$, which is a contradiction, since $|E(G)| \leq \frac{\Delta(G)n}{2}$ holds in every graph $G$. 
We may thus assume that $n$ is odd.

Now, let $c$ be a proper $\chi'(G)$-edge coloring of $G$ with $\Delta(G)+1$ colors and suppose that $\es(G) = \frac{n-1}{2}$. Then by Remark~\ref{r:C1} every color class of $c$ contains exactly $\frac{n-1}{2}$ edges. Thus $|E(G)|=(\Delta(G)+1)\frac{n-1}{2}$. We note that $|E(G)| \leq \frac {n \Delta(G)}{2}$, which implies $\Delta(G) \geq n-1$, and hence $G$ is isomorphic to $K_n$. The reverse direction follows from Lemma~\ref{l:complete}.
\end{proof}
 
\begin{lemma}\label{l:evenClass1}
If $G$ is a Class $1$ graph of even order $n$, then $\es(G)=\frac{n}{2}$ if and only if $G$ is a regular graph.
\end{lemma}
\begin{proof}
Let $G$ be a Class 1 graph of even order $n$, and suppose that $\es(G)=\frac{n}{2}$. If $c$ is a proper $\chi'(G)$-edge coloring of $G$ with $\Delta(G)$ colors, then each color class of $c$ must have exactly $\frac{n}{2}$ edges. Hence $|E(G)|=\Delta(G) \frac{n}{2}$, and $G$ is a $\Delta(G)$-regular graph. 

Conversely, let $G$ be a $\Delta(G)$-regular graph of Class 1. In order to reduce the chromatic index (to $\Delta(G)-1$), one must remove a set of edges $F$ from $G$ in such a way that $F$ covers all vertices of $G$. Hence $\es(G)\ge \frac{n}{2}$, and by Lemma~\ref{l:bound}, $\es(G)\le \frac{n}{2}$, which completes the proof.
\end{proof}

%

Since there are no Class 1 regular graphs of odd order, we get the following.

\begin{corollary}\label{cor:regClass1}
Let $G$ be a Class $1$ graph of order $n$. If $G$ is regular, then $\es(G)=\frac{n}{2}$.
\end{corollary}

\begin{lemma}\label{l:oddClass1}
If $G$ is a Class $1$ graph of odd order $n$, then $\es(G)= \lfloor \frac{n}{2} \rfloor$ if and only if $|E(G)|=\frac{(n-1)\Delta(G)}{2}$.
\end{lemma}
\begin{proof}
First, let $\es(G)=\frac{n-1}{2}$ and let $c$ be a proper $\chi'(G)$-edge coloring of $G$ with $\Delta(G)$ colors. Since $\es(G)=\frac{n-1}{2}$, each color class of $c$ has exactly $\frac{n-1}{2}$ edges. Hence $|E(G)|=\Delta(G)\frac{n-1}{2}$.

For the converse, let $G$ be a graph with $|E(G)|=\frac{(n-1)\Delta(G)}{2}$. Now, Lemma~\ref{l:bound} implies that $\es(G) \leq \frac{n-1}{2}$. Suppose that $\es(G) < \frac{n-1}{2}$. Let $F$ be a minimum mitigating set of $G$, that is, $\es(G)=|F| \leq \frac{n-3}{2}$, and let $G' = G \setminus F$. Hence 
\begin{equation}\label{eq:edges}
|E(G')|=|E(G)|-|F| \geq \frac{(n-1)\Delta(G)}{2}- \frac{n-3}{2}.
\end{equation} 
Since $F$ is a minimum mitigating set of $G$ and $\chi'(G)=\Delta(G)$, $\chi'(G')=\chi'(G)-1= \Delta(G)-1$. Thus there exists a proper edge coloring $c$ of $G'$ with $\Delta(G)-1$ colors. Since each color class can contain at most $\frac{n-1}{2}$ edges, we get $|E(G')| \leq (\Delta(G)-1)\frac{n-1}{2}$ which contradicts~(\ref{eq:edges}).
\end{proof}

Combining Theorem~\ref{p:extremeClass2}, and Lemmas~\ref{l:evenClass1} and \ref{l:oddClass1} we get the following characterization. 

\begin{theorem}\label{t:mainExtreme}
A graph $G$ of order $n$ has $\es(G)=\lfloor \frac{n}{2} \rfloor$ if and only if one of the following holds:
\begin{enumerate}
\item $G$ has odd order and is isomorphic to $K_n$;
\item $G$ is a Class $1$ regular graph of even order;
\item $G$ is a Class $1$ graph of odd order with $\frac{(n-1)\Delta(G)}{2}$ edges. 
\end{enumerate}
\end{theorem}

Next, we focus our attention on the graphs $G$ with $\es(G)=1$.

\begin{lemma}\label{l:singelton}
A graph $G$ has $es_{\chi'}(G)= 1$ if and only if there exists a proper $\chi'(G)$-edge coloring of $G$ with a singleton color class.
\end{lemma}
\begin{proof}
Let $es_{\chi'}(G)= 1$. This means that there exists an edge $e'$ such that $\chi'(G \setminus e')  =\chi'(G)-1$. Let $c'$ be a proper $(\chi'(G)-1)$-edge coloring of $G \setminus e'$ with colors $c_1,\ldots , c_k$ and $k=\chi'(G)-1$. Then $c:E(G) \to \{c_1,\ldots , c_{k+1}\}$ with $c(e)=c'(e)$ for any $e \in E(G) \setminus e'$ and $c(e')=c_{k+1}$ is a proper $\chi'(G)$-edge coloring of the edges of $G$ which contains a singleton color class. The reverse direction follows from Remark~\ref{r:C1} and Lemma~\ref{l:bound}.
\end{proof}

\begin{lemma}\label{prp:class1}
If $G$ is a graph of Class $1$ with $es_{\chi'}(G)=1$, then $G$ contains at most two vertices of degree $\Delta(G)$, and if there are two vertices of degree $\Delta(G)$, then they are adjacent.
\end{lemma}
\begin{proof}
Let $G$ be a graph of Class 1 with $es_{\chi'}(G)=1$. If $G$ contains more than two vertices of degree $\Delta(G)$ or two non adjacent vertices of degree $\Delta(G)$, then $\Delta(G \setminus e)=\Delta(G)$ for any $e \in E(G)$, and so $\chi'(G \setminus e) \geq \Delta(G \setminus e)=\Delta(G)= \chi'(G)$. Hence, $es_{\chi'}(G)\geq 2$, a contradiction.
\end{proof}

Next, we characterize graphs $G$ with $\es(G)=1$ among all connected regular graphs, showing that there are just two such families.

\begin{theorem}
\label{thm:reg1}
If $G$ is a connected regular graph, then $es_{\chi'}(G)=1$ if and only if $G$ is isomorphic to $K_{2}$ or $C_n$ with odd $n$.
\end{theorem}
\begin{proof}
Let $G$ be a connected regular graph of order $n$ with $\es(G)=1$. Hence it follows from Lemma~\ref{l:singelton} that there exists a proper $\chi'(G)$-edge coloring $c$ of the edges of $G$ with a singleton color class $C_1$. Suppose first that $\chi'(G)=\Delta(G) + 1$.  Hence $\vert C_{1}\vert+\cdots+\vert C_{\Delta(G)+1}\vert=\frac{n\Delta(G)}{2}$ and $|C_i| \leq \frac{n}{2}$ for any $i \in \{2,\ldots , \Delta(G)+1\}$. If $n$ is even, (since $\vert C_{1} \vert =1 $) the equality holds if and only if there exists $i \in \{2,\ldots , \Delta(G)+1\}$ such that $\vert C_{i}\vert=\frac{n}{2}-1$ and $\vert C_j \vert=\frac{n}{2}$ for any $j\in \{2,\ldots , \Delta(G)+1\}\setminus\{i\}$. So, we have $\Delta(G)-1$ disjoint perfect matchings. Since $G$ is $\Delta(G)$-regular, $\chi'(G)=\Delta(G)$, a contradiction.
If $n$ is odd, then for every $i\in \{2,\ldots ,\Delta(G)+1\}$, $\vert C_{i}\vert\leq\frac{n-1}{2}$. Hence, $\frac{n\Delta(G)}{2}=\vert C_{1}\vert+\cdots+\vert C_{\Delta(G)+1}\vert \leq 1+\frac{n-1}{2}+\cdots+\frac{n-1}{2}$ which holds when $\Delta(G)\leq 2$. We infer that $G$ is isomorphic to an odd cycle. Now, by Lemma~\ref{prp:class1} the only regular graph $G$ of Class 1 with $es_{\chi'}(G)=1$ is $K_{2}$, which completes the proof, since the converse is clear. 
\end{proof}

In the rest of this section, we focus on the computational complexity of determining whether a given graph achieves one of the extremal values.

For a regular graph $G$ of even order Theorem~\ref{t:mainExtreme} implies that $\es(G)=\lfloor\frac {n}{2}\rfloor$ if and only if $G$ is of Class $1$.
Note that $3$-regular graphs are necessarily of even order. Hence we can combine this with Theorem~\ref{thm:holyer} about NP-completeness of determining whether a $3$-regular graph is of Class 1 or not. Notably, if the problem of determining whether a graph $G$ has $\es(G)=\lfloor\frac {n}{2}\rfloor$ is polynomial, this would imply that there is a polynomial algorithm to determine whether a $3$-regular graph is of Class 1. We derive the following result.

\begin{corollary}
Given a graph $G$, it is NP-hard to determine whether $\es(G)=\lfloor\frac {n}{2}\rfloor$, even if $G$ is restricted to $3$-regular graphs. 
\end{corollary}

We prove a similar result for graphs with the chromatic edge stability index equal to $1$, but need to use some additional arguments. In particular, we will use the result of Fournier, Lemma~\ref{Fournier}, that if $Core(G)$ is a forest, then $G$ is of Class 1.


\begin{lemma}
\label{lem:regularNP}
Let $G \neq K_n$ be an $r$-regular graph of order $n$ and let $G'$ be obtained from $G$ by adding an arbitrary edge that is not in $E(G)$. Then
$\es(G')=1$ if and only if $G$ is of Class $1$.
\end{lemma}
\begin{proof}
We start by proving the following claim.

\begin{claim}\label{G'class1}
$\chi'(G')=\Delta(G)+1$.
\end{claim}
\begin{claimproof}
Since $\Delta(G')=\Delta(G)+1$ and exactly two vertices of $G'$ have degree $\Delta(G')$, $Core(G')$ is a forest. Hence by Lemma~\ref{Fournier}, $G'$ is of Class 1 and $\chi'(G')=\Delta(G')=\Delta(G)+1$.
\end{claimproof}

Suppose first that $\es(G')=1$. Hence there exists $e \in E(G')$ such that $\chi'(G' \setminus e)=\chi'(G')-1=\Delta(G)$. Hence  $\Delta(G'\setminus e) \leq \Delta(G)=\Delta(G')-1$. Since with the removal of one edge, the maximum degree can decrease by at most one, $\Delta(G' \setminus e)=\Delta(G')-1 = \Delta(G)$ and hence $e=xy$, where $xy$ is the edge in $E(G')\setminus E(G)$. Consequently $G' \setminus e=G$. Hence  $\chi'(G)=\chi'(G' \setminus e)=\chi'(G')-1=\Delta(G)$ and thus $G$ is of Class 1.

For the converse, suppose that $\chi'(G)=\Delta(G)$. Since by Claim~\ref{G'class1}, $\chi'(G')=\Delta(G')=\Delta(G)+1$, $F=\{xy\}$ is a minimum mitigating set of $G'$ and hence $\es(G')=1$.
\end{proof}

By Lemma~\ref{lem:regularNP}, we infer that the existence of a polynomial algorithm to determine whether a graph $G$ has $\es(G)=1$ would imply that one can determine by polynomial algorithm whether a regular graph is of Class 1. We again invoke the result of Holyer, Theorem~\ref{thm:holyer}, and derive the following. 

\begin{corollary}
Given a graph $G$, it is NP-hard to determine whether $\es(G)=1$, even if $G$ is restricted to graphs with $\Delta(G)=4$ and $Core(G)=K_2$. 
\end{corollary}

\section{Graphs with $es_{\chi'}(G)= 2$ }\label{sec:2}

In this section, we prove that Conjecture~\ref{conj1} is true for graphs with the chromatic edge stability index equal to $2$. We start by showing that in $r$-regular graphs, where $r\ne 4$, the only minimum mitigating sets are $2$-matchings.

\begin{theorem}
If $G$ is an $r$-regular graph, where $r\neq 4$, and $es_{\chi'}(G)=2$, then every minimum mitigating set of size $2$ is a $2$-matching of $G$.
\end{theorem}

\begin{proof}
If $G$ is of Class 1, then by Corollary~\ref{cor:regClass1} we have $n=4$. Hence $G$ is either $K_{4}$ or $C_{4}$ or $K_2 \cup K_2$, so the statement of the theorem holds. 

Now, let $G$ be a graph of Class 2. Assume that there exists a set of edges $S=\{xy,xz\}$ such that $\chi'(G \setminus S)=\chi'(G)-1=r$. Let $H=G \setminus S$ and $C_{1},\ldots , C_{r}$ be the color classes in a proper $r$-edge coloring of $H$. We distinguish two cases with respect to the parity of $n$.\\

\noindent {\bf{Case 1.}} If $n$ is odd, then for every $i\in [r]$ we have $\vert C_{i}\vert\leq\frac{n-1}{2}$. Hence, $\frac{nr}{2}=|E(G)|=2+|E(H)|\le 2+r(\frac{n-1}{2})$, and so $r\leq 4$. As $r\neq 4$ and $n$ is odd, the only possible value for $r$ is 2. Hence, $G$ is a disjoint union of cycles and since $es_{\chi'}(G)=2$, exactly two of them are odd. Therefore $G$ has even number of vertices, a contradiction.\\

\noindent {\bf{Case 2.}} If $n$ is even, then for every $i\in [r]$ we have $\vert C_{i}\vert\leq\frac{n}{2}$. Since $\vert E(H)\vert=\frac{nr}{2}-2$, the only possible sizes of the color classes of $H$ are $\vert C_{1}\vert=\frac{n}{2}-2$ and $\vert C_{i}\vert =\frac{n}{2}$ for $i\in\{2,\ldots,r\}$, or $\vert C_{1}\vert=\vert C_{2}\vert=\frac{n}{2}-1$ and $\vert C_{i}\vert =\frac{n}{2}$ for $i\in\{3,\ldots,r\}$. If $\vert C_{1}\vert=\frac{n}{2}-2$, then the edges of $G$ which are not in $H$ are two disjoint edges, a contradiction. Now, let $\vert C_{1}\vert=\vert C_{2}\vert=\frac{n}{2}-1$. Then each $C_{i}$, where $i\in\{3,\ldots,r\}$, is an $\frac{n}{2}$-matching and each of the colors $c_{1}$ and $c_{2}$ does not appear in exactly two vertices. As $d_{H}(x)=r-2$ and $d_{H}(y)=r-1$ and $d_{H}(z)=r-1$, with no loss of generality, $c_{1}$ does not appear in two vertices $x$ and $y$ and $c_{2}$ does not appear in two vertices $x$ and $z$. Now, add two edges $xy$ and $xz$ to $H$ and color $xy$ and $xz$ by $c_{1}$ and $c_{2}$, respectively, and keep the colors of other edges. This is obviously a proper $r$-edge coloring of $G$, which means $G$ is of Class 1, a contradiction.
\end{proof}

\begin{remark}
For every odd integer $n=2t+1\geq 5$, there exists a $4$-regular graph $G$ of order $n$ such that $\es(G)=2$ with a minimum mitigating set which is not a $2$-matching.
\end{remark}

To see this, note by~\cite[Theorem 9.21]{chart-2005} that $E(K_{2t+1})$ can be decomposed into $t$ Hamiltonian cycles. Consider the union of $2$ Hamiltonian cycles of $K_n$ and call this graph by $G_n$. Since $G_n$ is a 4-regular graph of odd order, $\chi'(G_n)=5$. If $v \in V(G_n)$ and $e$ and $e'$ are two edges incident with $v$ and they are contained in two different Hamiltonian cycles, then clearly, $\chi'(G_n \setminus \{e,e'\})=4$, because the edge set of $G_n \setminus \{e,e'\}$ can be decomposed into two Hamiltonian paths.

\begin{lemma}\label{l:mitigating2}
Let $G$ be a graph with $es_{\chi'}(G)=2$ and $S=\{xy,xz\}$ be a minimum mitigating set. Then there exists a minimum mitigating set which is a $2$-matching if one of the following holds:
\begin{enumerate}[(i)]
\item $yz\notin E(G)$.
\item $yz\in E(G)$ and there exists a proper $(\chi'(G)-1)$-edge coloring $c$ of $G \setminus S$, such that there is a color $c_{1}$ with $c_{1}\notin c(x)$ and $c(yz)\neq c_{1}$.
\end{enumerate}
\end{lemma}
\begin{proof}
Let $H=G \setminus S$. First assume that $(i)$ holds, that is, $yz \notin E(G)$. Let $c$ be a proper $(\chi'(G)-1)$-edge coloring of $H$. Since $d_{H}(x)\leq\Delta(G)-2$, there exists at least one color $c_{1}$ such that $c_{1}\notin c(x)$. If $c_{1}\notin c(y)$, then we add the edge $xy$ to $H$ and color $xy$ by $c_1$ to obtain $(\chi'(G)-1)$-edge coloring of $G \setminus xz$, a contradiction. So, there exists an edge $yy'$, $y'\neq x$, such that $c(yy')=c_{1}$. Then $c':E(G)\setminus\{xz,yy'\}$ with $c'(e)=c(e)$ for any $e \in  E(H)\setminus yy'$ and $c'(xy)=c_1$ is a proper $(\chi'(G)-1)$-edge coloring of the graph $G \setminus \{yy',xz\}$. Hence $\{yy',xz\}$ is a $2$-matching minimum mitigating set of $G$. 

Now, suppose that $(ii)$ holds, that is, $yz\in E(G)$ and there is a proper $(\chi'(G)-1)$-edge coloring $c$ of $G \setminus S$, with a color $c_{1}$ such that $c_{1}\notin c(x)$ and $c(yz)\neq c_{1}$. If $c_1 \notin c(y)$, then $\chi'(G \setminus xz) < \chi'(G)$, a contradiction. So, there exists an edge $yy'$, $y'\notin \{x,z\}$ such that $c(yy')=c_{1}$. Then there exists a proper $(\chi'(G)-1)$-edge coloring $c'$ of $G \setminus \{yy',xz\}$ such that $c'(e)=c(e)$ for any $e \in E(G) \setminus \{yy',xz, xy\}$ and $c'(xy)=c(yy')=c_1$. Hence $\{yy',xz\}$ is a minimum mitigating set of $G$ which is a $2$-matching. 
\end{proof}

\begin{theorem}\label{th:es2}
Let $G$ be a graph with $es_{\chi'}(G)=2$. Then there exists a $2$-matching mitigating set.
\end{theorem}
\begin{proof}
Assume that $S=\{xy,xz\}$ is a mitigating set and let $c$ be a proper $(\chi'(G)-1)$-edge coloring of $H=G \setminus S$. We distinguish two cases with respect to the type of $G$:\\

\noindent {\bf{Case 1.}} If $G$ is of Class 2, then $\chi'(H)\le\Delta(G)$ and $d_{H}(x)\leq\Delta(G)-2$. Hence there exist at least two colors $c_{1}$ and $c_{2}$ such that $c_1,c_2 \notin c(x)$. Now, if $H$ contains the edge $yz$ then, $c(yz)\neq c_{1}$ or $c(yz)\neq c_{2}$ and hence we get the result by Lemma~\ref{l:mitigating2}. We also get the result by Lemma~\ref{l:mitigating2} if the edge $yz$ does not exist.\\

\noindent {\bf{Case 2.}} If $G$ is of Class 1, then $\chi'(H)=\Delta(G)-1$. Since $d_{H}(x)\leq\Delta(G)-2$, there exists a color $c_{1} \notin c(x)$ in a proper $(\Delta(G)-1)$-edge coloring $c$. If there exists a color $c_{2}\neq c_{1}$ such that $c_{2}\notin c(x)$, we get the result using Lemma~\ref{l:mitigating2}. Assume that there exists exactly one color $c_{1}$ not appearing in $x$. If $yz \notin E(H)$ or if $c(yz)\neq c_{1}$, we get the result by Lemma~\ref{l:mitigating2}. Hence assume that $yz \in E(H)$ and $c(yz)=c_{1}$. Since the vertices $x, y$ and $z$ induce a $K_{3}$ in $G$, $\chi'(G)\geq 3$ and $\chi'(H)=\chi'(G)-1\geq 2$. Therefore there exists a color $c_{2} \in c(x)$. Now, we consider the path $P_{x}(c_{2},c_{1})$ in $H$. This path passes the vertices $y$ and $z$ or none of them. If $y,z \notin V(P_{x}(c_{2},c_{1}))$, then by Remark~\ref{r:path} we can switch the colors $c_{1}$ and $c_{2}$ in $P_{x}(c_{2},c_{1})$ and keep the color of the other edges to obtain a proper $(\chi'(G)-1)$-edge coloring of $H$ such that $c_1=c(yz)\neq c_{2}$ and $c_{2}\notin c(x)$, so by Lemma~\ref{l:mitigating2} we get the result. Now, assume that $y,z\in V(P_{x}(c_{2},c_{1}))$ and first the path meets the vertex $y$ and then meets the edge $yz$. Let $P$ be the $x,y$-subpath of $P_x(c_2,c_1)$ and let $Q$ be the subpath of $P_x(c_2,c_1)$ induced by the vertices that are not in $P$. Note that $x,y \notin V(Q)$. Define $c':E(G) \setminus \{yz,xz\}$ as $c'(e)=c(e)$ for any $e \in E(H) \setminus \{yz\}$ and $c'(xy)=c_1 \notin c(x)$ is a proper $(\chi'(G)-1)$-edge coloring of the graph $G \setminus \{xz,yz\}$ with $S'=\{xz,yz\}$ as the mitigating set of $G$. Since $c_1 \notin c'(z)$, $Q=P_z(c_2,c_1)$ which does not contain $x$ and $y$. If $V(Q)=\{z\}$ then Lemma~\ref{l:mitigating2} implies the result. Otherwise by Remark~\ref{r:path}, one can switch the colors $c_{1}$ and $c_{2}$ in $P_{z}(c_{2},c_{1})$ of $G\setminus S'$ and keep the color of other edges. Since $c_2 \notin c'(z)$ and $c_1=c'(xy)\neq c_2$, we get the result by Lemma~\ref{l:mitigating2}.
\end{proof}

\section{Graphs with $es_{\chi'}(G)= \lfloor \frac{n}{2} \rfloor-1$ }\label{sec:5}

In this section, we prove that Conjecture~\ref{conj1} holds for every graph $G$ of order $n$ with $es_{\chi'}(G)= \lfloor \frac{n}{2} \rfloor-1$. In addition, we characterize the graphs of Class $2$ with the chromatic edge stability index equal to $\lfloor \frac {n}{2} \rfloor -1$.

We start by the following auxiliary results that are useful in the proof of the main theorem.

\begin{remark}\label{coloringComplete}
If $n$ is odd and $V(K_n)=\{v_1,\ldots , v_n\}$, then there exists a proper $n$-edge coloring $c$ of $K_n$ such that for each $i \in [n]$, $c_i \notin c(v_i)$ and $c_j \in c(v_i)$ for any $j \neq i$. 
\end{remark}

\begin{theorem}\label{th:perfectMatching} {\rm \cite{hou-2011}}
Let $G$ be a connected $k$-regular graph of even order $n$. If $k \geq \frac {n}{3}$, then $G$ contains a perfect matching.
\end{theorem}

\begin{lemma}\label{l:esLessthenA}
If $G$ is a graph of Class $2$ and $H=Core(G)$, then $\es(G) \leq \alpha'(H)$.
\end{lemma}
\begin{proof}
Let $M$ be a maximum matching of $H$ and let $G'=G \setminus M$.
If $M$ is a perfect matching of $H$, then $\Delta(G') < \Delta(G)$ and hence $\chi'(G') \leq \Delta(G')+1 < \Delta(G)+1 = \chi'(G)$. Thus, $\es(G) \leq |M|= \alpha'(H)$.

If $M$ is not a perfect matching of $H$, then $\Delta(G') = \Delta(G)$. Clearly, $V(H) \setminus V(M)$ is an independent set. Thus, $Core(G')$ is a forest and by Lemma~\ref{Fournier}, $G'$ is of Class 1. Hence, $\chi'(G')=\Delta(G') = \Delta(G) < \chi'(G)$ and the proof is complete.
\end{proof}

\begin{lemma}\label{perfectMatchingCore}
If $G$ is of Class $2$ of order $n$ and $M$ is a perfect matching of $Core(G)$, then $H=G \setminus M$ is of Class $2$ and $\es(H) \geq \es(G)$.
\end{lemma}
\begin{proof}
Note that $\Delta(H)=\Delta(G)-1$ and $\chi'(H)=\chi'(G)-1$, thus $H$ is of Class 2. If $N$ is a mitigating set of $H$, then $N$ is a mitigating set of $G$, which implies that $\es(H) \geq \es(G)$.
\end{proof}

\begin{lemma}\label{l:regularUnionPerfectMatching}
Let $r$ be a positive integer. If $G$ is the graph obtained from $K_{2r+1} \cup K_{2r+1}$ by adding a perfect matching, then $G$ is of Class $1$.
\end{lemma} 

\begin{proof}
Let $M$ be a perfect matching added to $K_{2r+1} \cup K_{2r+1}$. By Remark~\ref{coloringComplete}, let $c$ be a proper $(2r+1)$-edge coloring of $K_{2r+1} \cup K_{2r+1}$ such that for each $uv \in M$ there exists a color $c_i$ not appearing in $u$ and $v$. So, we color $uv$ with $c_i$. Hence, $G$ is of Class 1.
\end{proof}

\begin{theorem}\label{th:notClass2}
Let $G$ be a graph of even order $n$ which is of Class $2$. Then $es_{\chi'}(G)= \frac{n}{2}-1$ if and only if $G \cong K_{1} \cup K_{n-1}$ or $G \cong K_{2m+1} \cup K_{2m+1}$, where $m$ is a positive integer.
\end{theorem}
\begin{proof}
Let $H= Core(G)$ and $\vert V(H)\vert=a$. It follows from Lemma~\ref{l:esLessthenA} that $\frac{n}{2}-1=es_{\chi'}(G)\leq \alpha'(H) \leq \frac{a}{2}$. Therefore $a\geq n-2$. We distinguish three cases with respect to $a$.\\

\noindent {\bf{Case 1.}} Let $a=n-2$. Since $\alpha'(H)=\frac{n}{2}-1$, there exists a perfect matching $M$ in $H$. If we choose an arbitrary set $S$ of $\alpha'(H)-1$ edges of $M$, then $Core(G\setminus S)$ is a forest. By Lemma~\ref{Fournier}, $G\setminus S$ is of Class 1 and $\chi'(G\setminus S)<\chi'(G)$. Hence, $es_{\chi'}(G) \leq \frac{n}{2}-2$, a contradiction.\\

\noindent {\bf{Case 2.}} Let $a=n-1$. Since $\alpha'(H)=\frac{n}{2}-1$, every maximum matching $M$ of $H$ contains all vertices of $H$ except one vertex, say $v_{1}\in V(H)$. Assume that there exists an edge $v_{2}v_{3}\in M$ such that $v_{1}v_2\notin E(H)$ or $v_1v_3 \notin E(H)$. Then $Core(G\setminus(M\setminus \{ v_{2}v_{3} \}))$ is a forest of order 3 and by Lemma~\ref{Fournier}, the graph $G\setminus(M\setminus \{ v_{2}v_{3}\})$ is of Class 1. Hence, $es_{\chi'}(G) \leq \frac{n}{2}-2$, a contradiction. Therefore $d_{H}(v_1)=a-1$. Let $M'=(M\setminus \{v_2v_3\}) \cup \{v_1v_2\}$, and note that $M'$ is another maximum matching, hence from the same reason as earlier, $d_{H}(v_3)=a-1$. By repeating this argument, we infer $H \cong K_{n-1}$. Let $V(G)\setminus V(H)= \{ u\}$. If $G$ is not connected, then $G \cong K_{1} \cup K_{n-1}$, as desired. If $G$ is connected, then there exists a vertex $w\in V(H)$ such that $uw\in E(G)$. Hence, $\Delta(G)=d(w)=n-1$. Since $H=Core(G)$, we conclude that, $ G \cong K_{n}$ and $a=n$, a contradiction.\\

\noindent {\bf{Case 3.}} Let $a=n$. In this case $G$ is an $r$-regular graph. Since $es_{\chi'}(G)=\frac{n}{2}-1$, in any proper $\chi'(G)$-edge coloring of $G$, each color class has size at least $\frac {n}{2}-1$. Since $\chi'(G)=r+1$, $(r+1)(\frac {n}{2}-1) \leq \frac {nr}{2}=\vert E(G) \vert$, which implies that $r \geq \frac {n}{2}-1$. If $G$ is not connected, then obviously $G \cong K_{\frac {n}{2}} \cup K_{\frac {n}{2}}$, where $\frac {n}{2}$ is odd because $G$ is of Class $2$, as desired.
Now, if $G$ is connected by Theorem~\ref{th:perfectMatching}, $G$ has a perfect matching $M_1$ (note that $r \geq \frac {n}{2}-1 \geq \frac {n}{3}$, because otherwise $n=2$ or $n=4$ which implies that $G$ is of Class $1$, a contradiction). Let $G_1=G \setminus M_1$. By Lemma~\ref{perfectMatchingCore}, $G_1$ is of Class 2 and $\es(G_1) \geq \frac {n}{2}-1$. Now, Theorem~\ref{p:extremeClass2} yields that $es_{\chi'}(G_1)=\frac{n}{2}-1$. If $G_1$ is a connected $r_1$-regular graph, then by the same argument $r_1 \geq \frac{n}{2}-1$ and by Theorem~\ref{th:perfectMatching}, it has a perfect matching $M_2$. By repeating this procedure in $l$ steps, we find a regular graph $G_l$ of Class 2 which is not connected such that $\es(G_l)=\frac{n}{2}-1$. Since $G_l$ is a regular graph, similar to the previous proof the degree of $G_l$ is at least $\frac {n}{2}-1$. Hence, $G_l \cong K_{\frac {n}{2}} \cup K_{\frac {n}{2}}$, where $\frac {n}{2}$ is odd. By Lemma~\ref{l:regularUnionPerfectMatching}, $G_{l-1}$ is of Class 1, a contradiction.\\

The reverse direction follows from Theorem~\ref{p:extremeClass2}.
\end{proof}

Now, Theorems~\ref{th:notClass2} and~\ref{p:extremeClass2} imply the following.

\begin{corollary}\label{co:notClass2}
If $G$ is a connected Class $2$ graph of even order $n$, then $\es(G) \leq \frac{n}{2}-2$.
\end{corollary}

\begin{lemma}\label{l:regular}
If $G$ is an $(n-3)$-regular graph of odd order $n$, then $G$ is of Class $2$ and $\es(G) =\frac{n-3}{2}$.
\end{lemma} 

\begin{proof}
By Lemma~\ref{Bondy}, $G$ is of Class 2 and hence by Theorem~\ref{p:extremeClass2}, $\es(G) \leq \frac{n-3}{2}$. By contradiction suppose that $\es(G) < \frac{n-3}{2}$. Therefore, there exists $M \subseteq E(G)$ such that $|M| \leq \frac{n-5}{2}$ and $\chi' (G \setminus M)=\Delta(G)=n-3$.
For each color class $C_i$, $|C_i| \leq \frac{n-1}{2}$. So $(n-3)(\frac{n-1}{2}) \geq |E(G\setminus M)| \geq \frac{n(n-3)}{2}-\frac{n-5}{2}$, a contradiction.
\end{proof}

\begin{lemma}\label{l:almostRegular}
If $G= \overline {K_2 \cup \ldots \cup K_2 \cup K_{1,2}}$ is a graph of odd order $n$, then $G$ is of Class $2$ and $\es(G) =\frac{n-3}{2}$.
\end{lemma} 

\begin{proof}
Note that all vertices of $G$ except one have degree $\Delta(G)=n-2$ and one vertex of $G$ has degree $n-3$. Therefore $|E(G)|=\frac{(n-1)(n-2)+(n-3)}{2} > \frac{(n-1)\Delta(G)}{2}$ and hence by Lemma~\ref{Bondy} $G$ is of Class 2. By Theorem~\ref{p:extremeClass2}, $\es(G) \leq \frac{n-3}{2}$. By contradiction suppose that $\es(G) < \frac{n-3}{2}$. Therefore, there exists $M \subseteq E(G)$ such that $|M| \leq \frac{n-5}{2}$ and $\chi' (G \setminus M)=\Delta(G)=n-2$.
For each color class $C_i$, $|C_i| \leq \frac{n-1}{2}$. So $(n-2)(\frac{n-1}{2}) \geq |E(G\setminus M)| \geq \frac{(n-1)(n-2)+(n-3)}{2}-\frac{n-5}{2}$, a contradiction.
\end{proof}

\begin{lemma}\label{l:almostRegular2}
Let $r$ be a positive integer. If we add a matching of size $2r+1$ to $K_{2r+1} \cup K_{2r+1} \cup K_1$ and call the resulting graph by $G$, then $G$ is of Class $1$.
\end{lemma} 

\begin{proof}
Let $M$ be a matching of size $2r+1$ added to $K_{2r+1} \cup K_{2r+1} \cup K_1$. By Remark~\ref{coloringComplete}, let $c$ be a proper $(2r+1)$-edge coloring of $K_{2r+1} \cup K_{2r+1} \cup K_1$ such that for each $uv \in M$, there exists color $c_i$ not appearing in $u$ and $v$. So, we color $uv$ with $c_i$. Hence, $G$ is of Class 1.
\end{proof}

\begin{lemma}\label{l:regular2}
Let $r$ be a positive integer. If we add a matching of size $3r+1$ to $K_{2r+1} \cup K_{2r+1} \cup K_{2r+1}$ and call the resulting graph by $G$, then $G$ is of Class $2$ and $\es(G) \leq r$.
\end{lemma} 

\begin{proof}
We have $|E(G)|=3 \frac {2r(2r+1)}{2}+(3r+1)> \frac {(6r+2)(2r+1)}{2}=\frac {(n-1) \Delta(G)}{2}$. Thus, by Lemma~\ref{Bondy}, $G$ is of Class $2$.
Now suppose that the vertex sets of three $K_{2r+1}$ are $X$, $Y$ and $Z$. Let $M_1$, $M_2$ and $M_3$ be all edges between $X$ and $Y$, $X$ and $Z$, $Y$ and $Z$, respectively and with no loss of generality assume that $|M_3| \leq |M_2| \leq |M_1|$. We claim that $M_3$ is a mitigating set. By Remark~\ref{coloringComplete}, let $c$ be a proper $(2r+1)$-edge coloring of $K_{2r+1} \cup K_{2r+1} \cup K_{2r+1}$ such that for each $uv \in M_1 \cup M_2$, there exists color $c_i$ not appearing in $u$ and $v$. So we color $uv$ with $c_i$. Hence, $G \setminus M_3$ is of Class 1. So, $M_3$ is a mitigating set of $G$ and $|M_3| \leq \lfloor \frac {M_1+M_2+M_3}{3} \rfloor=\lfloor \frac {3r+1}{3} \rfloor=r$ and the proof is complete.
\end{proof}

\begin{lemma}\label{l:disconnectedOddClass2}
Let $G$ be a graph of Class $2$ of odd order $n$ which is not connected. Then $\es(G) =\frac{n-3}{2}$ if and only if $G$ is isomorphic to $K_{n-2} \cup K_2$ or $K_{n-2} \cup \overline{K_2}$, for $n \geq 5$ or $K_{2m+1} \cup K_{2m+1} \cup K_1$ or $K_{2m+1} \cup K_{2m+1} \cup K_{2m+1}$, where $m$ is a positive integer.
\end{lemma}

\begin{proof}
By Theorem~\ref{p:extremeClass2}, it is not hard to see that all graphs given in the statement of the lemma are of Class 2, have odd order and chromatic edge stability index equal to $\frac{n-3}{2}$.

For the reverse direction, suppose that $G$ has $k$ components $H_1, \ldots , H_k$. For each $H_i$, $i \in [k]$, either $\chi'(H_i) \leq \Delta(G)$ or $\chi'(H_i)=\Delta(G)+1$ and by Theorem~\ref{p:extremeClass2}, $\es(H_i) \leq \frac{|V(H_i)|-1}{2}$. So for every $i \in [k]$, there exists $M_i \subset E(H_i)$ such that $\chi'(H_i \setminus M_i) \leq \Delta(G)$ and $|M_i| \leq \frac{|V(H_i)|-1}{2}$, and we may assume that $M_i$ is a minimum such set (in particular, if $\chi'(H_i)\le\Delta(G)$, then $M_i=\varnothing$). Clearly, $M=M_1 \cup M_2 \cup \cdots \cup M_k$ is a mitigating set and $\es(G)=\frac {n-3}{2} \leq |M|=|M_1|+\cdots +|M_k| \leq \frac{|V(H_1)|-1}{2}+ \cdots +\frac{|V(H_k)|-1}{2} \leq \frac {n-k}{2}$. So, $k \leq 3$  and now we consider two cases:\\

\noindent {\bf Case 1.} $k=3$. $\es(G)=\frac {n-3}{2} \leq |M|=|M_1|+|M_2|+|M_3| \leq \frac{|V(H_1)|-1}{2}+\frac{|V(H_2)|-1}{2}+\frac{|V(H_3)|-1}{2}=\frac {n-3}{2}$. So $|M_1|=\frac{|V(H_1)|-1}{2}$ and $|M_2|=\frac{|V(H_2)|-1}{2}$ and $|M_3|=\frac{|V(H_3)|-1}{2}$. If $\chi'(H_i) \leq \Delta(G)$, for $1\leq i \leq 3$, then $|M_i|=\frac{|V(H_i)|-1}{2}=0$ and hence $|V(H_i)|=1$. If $\chi'(H_i)= \Delta(G)+1$, for $1\leq i \leq 3$, then Theorem~\ref{p:extremeClass2} yields that $H_i \cong K_{\Delta(G)+1}$. So $H_i \cong K_1$ or $H_i \cong K_{\Delta(G)+1}$. Since $G$ is of Class 2, it is isomorphic to $K_{n-2} \cup \overline{K_2}$ or $K_{2m+1} \cup K_{2m+1} \cup K_1$ or $K_{2m+1} \cup K_{2m+1} \cup K_{2m+1}$, where $m$ is a positive integer, as desired.\\

\noindent {\bf Case 2.} $k=2$. If $|M_1| \leq \frac{|V(H_1)|-3}{2}$ or $|M_2| \leq \frac{|V(H_2)|-3}{2}$, then $|M| \leq \frac {n-4}{2} < \frac {n-3}{2}=\es(G)$, a contradiction. So $|M_i|=\frac{|V(H_i)|-2}{2}$ or $|M_i|=\frac{|V(H_i)|-1}{2}$, for $1\leq i \leq 2$. With no loss of generality suppose that $|M_1|=\frac{|V(H_1)|-1}{2}$ and $|M_2|=\frac{|V(H_2)|-2}{2}$. By Corollary~\ref{co:notClass2}, there is no connected graph of even order of Class 2 with $\es(G)=\frac{|V(H_2)|-2}{2}$. So, $\chi'(H_2) \leq \Delta(G)$ and then $|M_2|=\frac{|V(H_2)|-2}{2}=0$ and hence $|V(H_2)|=2$. Since $G$ is of Class 2, $H_1$ is of Class 2 and by Theorem~\ref{p:extremeClass2}, $H_1$ is a complete graph of odd order. Hence, $G$ is isomorphic to $K_{n-2} \cup K_2$, as desired.
\end{proof}

For edge $e$ denote $K_n \setminus e$ by $K_n^-$.

\begin{theorem}\label{th:oddClass2}
Let $G$ be a graph of Class $2$ of odd order $n$. Then $\es(G) =\frac{n-3}{2}$ if and only if $G$ is isomorphic to $K_n^-$ or $K_{n-2} \cup K_2$ or $K_{n-2} \cup \overline{K_2}$ or $\overline {K_2 \cup \ldots \cup K_2\cup K_{1,2}}$ or an $(n-3)$-regular graph, for $n \geq 5$ or $K_{2m+1} \cup K_{2m+1} \cup K_1$ or $K_{2m+1} \cup K_{2m+1} \cup K_{2m+1}$, where $m$ is a positive integer.
\end{theorem}

\begin{proof}
By Theorem~\ref{p:extremeClass2} and Lemmas~\ref{l:regular}, \ref{l:almostRegular} and~\ref{l:disconnectedOddClass2}, it is not hard to see that all graphs given in the statement of the theorem are of Class 2, have odd order with chromatic edge stability index equal to $\frac{n-3}{2}$.

If $G$ is not connected, then by Lemma~\ref{l:disconnectedOddClass2}, it is isomorphic to $K_{n-2} \cup K_2$ or $K_{n-2} \cup \overline{K_2}$, for $n \geq 5$ or $K_{2m+1} \cup K_{2m+1} \cup K_1$ or $K_{2m+1} \cup K_{2m+1} \cup K_{2m+1}$, where $m$ is a positive integer, as desired.

Now, suppose that $G$ is connected and let $H=Core(G)$.
It follows from Lemma~\ref{l:esLessthenA} that $\frac{n-3}{2}=\es(G) \leq \alpha'(H) \leq \frac{a}{2}$, where $a=|V(H)|$ which implies $a \geq n-3$. Now, we consider four cases:\\

\noindent {\bf Case 1.} $a=n-3$. Since $\alpha'(H) \geq \frac{n-3}{2}$, there exists a perfect matching $M$ in $H$. If we choose an arbitrary set $F$ of $\alpha'(H)-1$ edges from $M$, then $Core(G \setminus F)$ is a forest and hence by Lemma~\ref{Fournier}, $\chi'(G \setminus F) = \Delta(G \setminus F)=\Delta(G) < \chi'(G)$. Therefore, $\es(G) \leq \alpha'(H)-1=\frac{n-3}{2}-1$, a contradiction.\\

\noindent {\bf Case 2.} $a=n-2$. Let $V(H)=\{v_1,\ldots , v_{n-2}\}$, and let $M$ be a maximum matching of $H$. Since $\alpha'(H)= \frac{n-3}{2}$, we have $|V(H)\setminus V(M)|=1$, and let $v_1\in V(H)\setminus V(M)$. Assume that there exists an edge $v_{2}v_{3}\in M$ such that $v_{1}v_2\notin E(H)$ or $v_1v_3 \notin E(H)$. Then $Core(G \setminus (M \setminus \{ v_{2}v_{3} \}))$ is a forest of order 3 and by Lemma~\ref{Fournier}, the graph $G \setminus (M \setminus \{ v_{2}v_{3}\})$ is of Class 1. Hence, $es_{\chi'}(G)<\frac{n-3}{2}$, a contradiction. Therefore $d_H(v_1)=a-1$. Let $M'=(M\setminus \{v_2v_3\}) \cup \{v_1v_2\}$ which is obviously another maximum matching of $H$ and by the same method $d_H(v_3)=a-1$. Now, by the same manner we infer that $H\cong K_{n-2}$.

Now, let $V(G) \setminus V(H) = \{x,y\}$. Since $G$ is connected, there is at least one edge between $\{x,y\}$ and $V(H)$. Hence there exists $w \in V(H)$ such that $d_G(w) \geq n-2$. Since all vertices of $H$ have the same degree in $G$, then $d_G(v) \geq n-2$, for any $v \in V(H)$. Suppose that there exists $w \in V(H)$ that is adjacent to both $x$ and $y$. Therefore $d_G(v)=n-1$ for all $v \in V(H)$. Since $x,y \notin Core(G)$, $x$ and $y$ are not adjacent. Thus $G \cong K_n^-$.

Therefore, it remains to consider the case when $\Delta(G)=n-2$ and each vertex $v \in V(H)$ is adjacent to exactly one vertex from $\{x,y\}$.  By Remark~\ref{coloringComplete}, there exists a proper edge coloring $c'$ of $H=K_{n-2}$ with $\Delta(H)+1=n-2$ colors such that for each $v_i \in V(H)$,  $c_i \notin c'(v_i)$, and $c_j \in c'(v_i)$ for any $j \neq i$. Now, we extend this coloring to a proper edge coloring $c$ of $E(G)$ as follows. If $e\in E(H)$, then define $c(e)=c'(e)$. If $v_i \in V(H)$ is adjacent to $x$ in $G$, then define $c(xv_i)=c_i$, and if $v_i \in V(H)$ is adjacent to $y$ in $G$, then define $c(yv_i)=c_i$. Since $\chi'(G) > \Delta(G)$, $xy \in E(G)$. Indeed, if $xy \notin E(G)$, $c$ is a proper edge coloring of $G$ with $\Delta(G)$ colors. Hence $G$ is a graph with $\es(G)=1$. Since $\es(G)=\frac{n-3}{2}$ we deduce that $n=5$ and so $G \cong \overline {K_2 \cup K_{1,2}}$. Hence $a=4=n-1$, a contradiction.\\

\noindent {\bf Case 3.} $a=n$. Then $G$ is a connected $r$-regular graph. We start by showing that $\alpha' (G) = \frac {n-1}{2}$. By contradiction assume that $\alpha' (G) \leq \frac {n-3}{2}$. Let $M$ be a matching of size $\alpha'(G)$. Let $e \in M$ and $S=V(G) \setminus V(M)$. Note that $S$ is an independent set. If for every $u \in S$, there is at most one edge between $u$ and vertices of $e$, then $Core(G \setminus (M \setminus e))$ is a forest and so by Lemma~\ref{Fournier}, $M \setminus e$ is a mitigating set of size at most $\frac {n-5}{2}$, a contradiction. So there exists $u \in S$ such that $u$ is adjacent to both endvertices of $e$. Every $u' \in S\setminus\{u\}$ is not adjacent to an endvertex of $e$, since $M$ is a maximum matching. Therefore, $Core(G \setminus (M \setminus e))$ is a unicyclic graph, thus it enjoys the conditions of Lemma~\ref{Akbari}. Hence $M \setminus \{e\}$ is a mitigating set of size at most $\frac {n-5}{2}$, a contradiction. Thus $\alpha' (G)=\frac {n-1}{2}$.
$\vspace{2mm}$

Now, we show that $r \geq \frac {n-1}{2}$. Let $V(G) \setminus V(M)=\{u\}$. If $d(u)=r<\frac {n-1}{2}$, then there exists $ab \in M$ such that $u$ is adjacent to no $a$ and $b$. For every $cd \in M\setminus\{ab\}$, $G[\{a,b,c,d\}]$ has at least 4 edges, because otherwise $Core(G \setminus (M \setminus \{ ab , cd \}))$ is a unicyclic graph or a tree, and so it enjoys the conditions in Lemma~\ref{Akbari}. Hence $M \setminus \{ ab, cd \}$ is a mitigating set of size $\frac {n-5}{2}$, a contradiction. So $G[\{a,b,c,d\}]$ has at least 4 edges and hence, $d(a)+d(b) \geq 2 \frac {n-3}{2}+2$. So $r \geq \frac {n-1}{2}$ since $d(a) \geq \frac {n-1}{2}$ or $d(b) \geq \frac {n-1}{2}$.
$\vspace{2mm}$

Now, let $u \in V(G)$. Clearly, $|V(G) \setminus N[u]|$ is even. If $|V(G) \setminus N[u]|=0$, then $G=K_n$ and by Theorem~\ref{p:extremeClass2}, $\es(G)=\frac{n-1}{2}$, a contradiction. Else if $|V(G) \setminus N[u]|=2$, then $G$ is an $(n-3)$-regular graph, where $n \geq 5$, as desired. Thus $|V(G) \setminus N[u]| \geq 4$. Now, we consider two cases: either $G[V(G) \setminus N[u]]$ has at least one edge or $V(G) \setminus N[u]$ is an independent set. \\

\noindent {\bf Subcase 3.1.} In this case, $G[V(G) \setminus N[u]]$ has at least one edge. Since $G$ is connected, there exists $vw,vy \in E(G)$ such that $w \in N(u)$ and $v,y \in V(G) \setminus N[u]$. Let $M'$ be a maximum matching of $G[V(G) \setminus N[u]]$ containing $vy$. If $M'$ saturates all vertices of $V(G) \setminus N[u]$, then let $e' \in M'\setminus\{vy\}$ and define $M=M' \setminus \{e'\}$, otherwise define $M=M'$. Let $T$ be the set of all edges incident with $u$. Hence, $Core(G \setminus (M \cup T))$ is a forest and so by Lemma~\ref{Fournier}, $M \cup T$ is a mitigating set.

By Lemma~\ref{BalancedColoring}, let $c$ be a balanced edge coloring of $G \setminus (M \cup T)$. Clearly, $c$ uses exactly $r$ colors. Now, there exists a color $c_i$ in $c$ such that $c_i$ appears in all vertices of $V(G) \setminus \{u\}$, because otherwise for each $j \in [r]$, $|C_j| \leq \frac {n-3}{2}$ and hence $\frac {n-3}{2}r \geq |E(G \setminus (M \cup T))| \geq \frac {(n-1)(r-1)+2}{2}$. This yields $r \leq \frac {n-3}{2}$, a contradiction. Since $c$ is a balanced edge coloring, for each $j \in [r]$, $c_j$ appears in all vertices of $V(G) \setminus \{u\}$ or not appearing in exactly two vertices of $V(G) \setminus \{u\}$.

Define $f_c:V(G) \to \{c_1,\ldots , c_{r},\infty\}$ such that if $d_{G \setminus (M \cup T)}(z)=r-1$, then $f_{c}(z)=c_i$, where $c_i \notin c(z)$, otherwise $f_{c}(z)=\infty$. Let $L_{c}$ be a maximum subset of $N(u)$ such that the restriction of $f_c$ to $L_c$ is a one-to-one function. Let $N_{c}=N(u) \setminus L_{c}$. (By $N_c'$, respectively $L_c'$, we denote the set of edges incident with $u$ and with a vertex of $N_c$, respectively $L_c$.)
Now, let $c'$ be a proper $r$-edge coloring of $G \setminus (N_c' \cup M)$ such that $c'(uu')=f_c(u')$, for each $uu' \in L_c'$ and $c'(e)=c(e)$, for each $e \in E(G \setminus (T \cup M))$. So, $N_{c}' \cup M$ is a mitigating set of $G$.

For each $a \in N(u)$, there is at most one vertex $b \in N(u) \setminus \{a\}$, such that $f_{c}(a)=f_{c}(b)$, because $c$ is a balanced edge coloring. So, $|L_{c}| \geq \frac{r}{2}$ and $|N_{c}| \leq \frac {r}{2}$. On the other hand we have $|M| \leq \frac {n-r-3}{2}$. Hence, $\es(G) \leq |N_{c}' \cup M| \leq \frac {n-3}{2}=\es(G)$. Thus, $|N_{c}|=\frac{r}{2}$, $|M|= \frac {n-r-3}{2}$ and we conclude that $f_{c}(a) \neq f_{c}(b)$, for every $a \in N(u)$, $b \in V(G) \setminus N[u]$.

Now, with no loss of generality suppose that $f_{c}(w)=c_1$ and $f_{c}(v)=c_2$ and $c_1 \neq c_2$. So, there exists $w' \in N(u)\setminus\{w\}$ and $v' \in V(G) \setminus N[u]$, where $v'\ne v$, such that $f_{c}(w')=c_1$ and $f_{c}(v')=c_2$ respectively. By Remark~\ref{r:path}, let $c^1$ be a proper $r$-edge coloring of $G \setminus (M \cup T)$ obtained by replacing $c_1$ and $c_2$ in $P_v(c_1,c_2)$ of $c$. Clearly, $P_v(c_1,c_2)$ in coloring $c$, ends in $w$ or $w'$ or $v'$. If it ends in $w$ or $w'$, then $|L_{c^1}|=|L_{c}|+1= \frac{r}{2}+1$. So $M \cup N'_{c^1}$ is a mitigating set of size $\frac {n-5}{2}$, a contradiction.

If $P_v(c_1,c_2)$ ends in $v'$, then suppose that $c^1(vw)=c_3$. Let $c^2$ be a proper $r$-edge coloring of $G \setminus (M \cup T)$ such that $c^2(vw)=c_1$ and $c^2(e)=c^1(e)$, for any $vw \neq e \in E(G \setminus (M \cup T))$. Clearly, for any $a \in V(G)$, then $f_{c^2}(a) \neq c_2$. If there is no $a \in N(u)\setminus\{w\}$ such that $f_{c^2}(a)=c_3$, then $|L_{c^2}|=|L_{c}|+1= \frac{r}{2}+1$ and so $M \cup  N'_{c^2}$ is a mitigating set of size $\frac {n-5}{2}$, a contradiction. Otherwise, there are exactly two vertices $a,b \in N(u)\setminus\{w\}$ such that $f_{c^2}(a)=f_{c^2}(b)=c_3$ and for any $d \in V(G) \setminus N[u]$, different from $v$, we have $f_{c^2}(d) \neq c_3$, while for all $d \in V(G)$, $f_{c^2}(d) \neq c_2$. So, $P_v(c_2,c_3)$ ends in $a$ or $b$ or $w$. Let $c^3$ be a proper $r$-edge coloring of $G \setminus (M \cup T)$ by replacing $c_2$ and $c_3$ in $P_v(c_2,c_3)$ of $c^2$. Now, there is exactly one vertex $h \in N(u)$ such that $f_{c^3}(h)=c_2$. Hence, $|L_{c^3}|=|L_{c}|+1= \frac{r}{2}+1$ and $M \cup (T \setminus L'_{c^3})$ is a mitigating set of size $\frac{r}{2} \leq \frac {n-5}{2}$, a contradiction.\\

\noindent {\bf Subcase 3.2.} In this case $G[V(G) \setminus N[u]]$ is an independent set. Clearly, $Core(G \setminus T)$ is a forest and so by Lemma~\ref{Fournier}, $T$ is a mitigating set (where $T$ is again the set of all edges incident with $u$). By the same method, there is $N \subseteq T$, $|N| \leq \frac{r}{2}$ such that $N$ is a mitigating set of size $\frac {r}{2} \leq \frac {n-5}{2}$, a contradiction.\\

\noindent {\bf Case 4.} $a=n-1$. We show that $|\alpha' (Core(G))| = \frac {n-1}{2}$. Let $M$ be a maximum matching of $Core(G)$. By Lemma~\ref{Fournier}, $M$ is a mitigating set. So $|M| \geq \frac {n-3}{2} = \es(G)$. To get a contradiction, assume that $|M|= \frac {n-3}{2}$. Let $V(Core(G)) \setminus V(M)= \{ a,b \}$. Clearly, $a$ is not adjacent to $b$ so for every $cd \in M$ there are at most 3 edges in $G[a,b,c,d]$, because otherwise $M$ is not maximum matching of $Core(G)$. Hence, $Core(G \setminus (M \setminus cd))$ is a unicyclic graph or a tree, and is not a disjoint union of cycles. Since $G$ is connected, by Lemma~\ref{Akbari}, $M \setminus cd$ is a mitigating set of size at most $\frac {n-5}{2}$, a contradiction.

Thus $G$ has a matching say $M_1$, which saturates $Core(G)$. Let $G_1=G \setminus M_1$. By Lemma~\ref{perfectMatchingCore}, $G_1$ is of Class 2 and $\es(G_1) \ge \es(G)= \frac {n-3}{2}$.
Since $G_1$ is not a complete graph, then by Theorem~\ref{p:extremeClass2}, $\es(G)=\frac{n-3}{2}$. If $G_1$ is connected and $|Core(G_1)|=n-1$, then similar to the previous proof, $G_1$ has a matching $M_2$ which saturates $Core(G_1)$. By repeating this procedure in $l$ steps, we find a Class 2 regular graph $G_l$ such that $\es(G_l)=\frac{n-3}{2}$ or a Class 2 graph $G_l$ which is not connected such that $|Core(G_l)|=n-1$ and $\es(G_l)=\frac{n-3}{2}$. By Case 3 and Lemma~\ref{l:disconnectedOddClass2}, $G_l$ is an $(n-3)$-regular graph, for $n \geq 5$ or $G_l \cong K_{2m+1} \cup K_{2m+1} \cup K_{2m+1}$ or $G_l \cong K_{2m+1} \cup K_{2m+1} \cup K_1$, where $m$ is a positive integer. Note that $G_{i-1}=G_i \cup M_i$, for $i=1, \ldots , l$, where $G_0=G$. If $G_l$ is an $(n-3)$-regular graph, then $l=1$ and so $G \cong \overline {K_2 \cup \ldots \cup K_2 \cup K_{1,2}}$. If $G_l \cong K_{2m+1} \cup K_{2m+1} \cup K_1$, then by Lemma~\ref{l:almostRegular2}, $G_{l-1}$ is of Class 1, a contradiction.  If $G_l \cong K_{2m+1} \cup K_{2m+1} \cup K_{2m+1}$, then by Lemma~\ref{l:regular2}, $\es(G_{l-1}) < \frac{n-3}{2}$, a contradiction.
\end{proof}

We say that a graph $G$ of order $n$ is {\em almost regular} if it has $n-1$ vertices of the same degree.

\begin{corollary}\label{col:oddClass2}
Let $G$ be a connected non-regular and not almost regular graph of Class $2$ of odd order $n$. Then $\es(G) =\frac{n-3}{2}$ if and only if $G$ is isomorphic to $K_n^-$, for $n \geq 5$.
\end{corollary}

In our next result we prove that Conjecture~\ref{conj1} holds also for graphs $G$ with $\es(G)=\lfloor \frac{n}{2} \rfloor -1$.

\begin{theorem}
\label{cor:n/2-1}
If $G$ is a graph of order $n$ with $\es(G)=\lfloor \frac{n}{2} \rfloor -1=k$, then there exists a $k$-matching mitigating set.
\end{theorem}  
\begin{proof}
First, we prove that the result holds for a graph $G$ of Class 2. Suppose that there does not exist a $k$-matching mitigating set in $G$.
If $n$ is even, then $|C_i|=\frac {n}{2}$, for each color class $C_i$. So $|E(G)|=\frac {n}{2}(\Delta(G)+1)$, a contradiction. If $n$ is odd, then $|C_i|=\frac {n-1}{2}$, for each color class $C_i$. So $|E(G)|=\frac{n-1}{2}(\Delta(G)+1)=\frac{(n-1)\Delta(G)+(n-1)}{2}$. This is possible only if $G$ is isomorphic to $K_n$, which is a contradiction as $\es(K_n)=\frac{n-1}{2}$ by Lemma~\ref{l:complete}. 

Now, let $G$ be of Class 1. Let $c$ be a proper edge coloring of $G$ using $\chi'(G)=\Delta(G)$ colors. Suppose that there does not exist a $k$-matching mitigating set in $G$. If $n$ is even, then $|C_i|=\frac {n}{2}$, for each color class $C_i$. Hence $|E(G)|=\frac{n}{2}\Delta(G)$, which implies that $G$ is a regular graph. By Corollary~\ref{cor:regClass1}, $\es(G)=\frac{n}{2}$, a contradiction. If $n$ is odd, then $|C_i|=\frac {n-1}{2}$, for each color class $C_i$. Hence $|E(G)|=\frac{n-1}{2}\Delta(G)$. Then it follows from Theorem~\ref{t:mainExtreme} that $\es(G)=\frac{n-1}{2}$, a contradiction. 
\end{proof}

\section{The chromatic edge stability index in bipartite graphs}
\label{sec:bip}

In this section we show that Conjecture~\ref{conj1} holds for bipartite graphs.

\begin{theorem}
Let $G$ be a bipartite graph with $es_{\chi'}(G)=k$. Then there exists a $k$-matching mitigating set.
\end{theorem}  
\begin{proof}
We prove by induction on $k$. 
For $k=1$, the assertion is trivial. Assume that the result holds for $k$ and $G$ is a bipartite graph with $\es(G)=k+1$ and $M$ is a mitigating set such that $|M|=k+1$. We choose an arbitrary edge $xy \in M$. By assumption, there exists a $k$-matching mitigating set $M'$ of $G \setminus xy$. If $\{x,y\} \cap V(M')= \varnothing$, then $M' \cup \{xy\} $ is a $(k+1)$-matching mitigating set of $G$, as desired.
Otherwise, let $M''=M' \cup \{xy\} $ and by Theorem~\ref{thm:konig} (K\H onig's theorem), there is a proper $(\Delta(G)-1)$-edge coloring $c$ of $G \setminus M'' $. Now, we consider two cases:\\

\noindent {\bf{Case 1.}} $\{x,y\} \cap V(M')=\{x\}$. Assume that $xz \in M'$. Suppose $x$ is in the partite set $A$,
and $y$, $z$ are in the partite set $B$. It is clear that there exists a color $c_{1}$ such that $c_{1} \notin c(x)$.
Let $P$ be the longest alternating path in $G$ starting at $x$, traverse $xy$ and whose edges are consecutively contained in $M''$ and $C_{1}$ ($C_{1}$ is the related color class to $c_{1}$).
Assume that the last vertex of $P$ is $u$. If $u \in B$, then the set of all edges of $P$ in $C_{1}$ union $M'' \setminus E(P)$, forms a mitigating set for $G$ of size $k$, a contradiction. To see this, color all edges of $E(P) \cap M''$ by $c_{1}$ and for each edge of $E(G) \setminus (M'' \cup E(P))$ keep its color in $c$. Now, if $u \in A$, then apply the same method to obtain a mitigating set, which is a $(k+1)$-matching of $G$.\\

\noindent {\bf{Case 2.}} $\{x,y\} \cap V(M')=\{x,y\}$. Since $G \setminus M''$ has a proper $(\Delta (G) - 1)$-edge coloring, we have $\Delta (G \setminus M'') \leq \Delta (G) - 1$. Note that $d_{G \setminus M''}(x) \leq \Delta (G) - 2$ and $d_{G \setminus M''}(y) \leq \Delta (G) - 2$. So $\Delta (G \setminus M') \leq \Delta (G) - 1$. Thus by Theorem~\ref{thm:konig}, $M'$ is a mitigating set of $G$ of size $k$, a contradiction.
 
\end{proof}

\section{Concluding remarks}

The main purpose of this paper is to introduce the new invariant measuring the effect of edge removal on proper edge colorings, which we call the chromatic edge stability index.  In doing so, we encounter a natural problem, which we posed as Conjecture~\ref{conj1}, and which claims that there exists a smallest edge set $F$, $F\subset E(G)$, whose removal results in a graph with smaller chromatic index than $G$, such that $F$ is a matching. Several results of this paper confirm the truth of the conjecture in some classes of graphs. Lemma~\ref{l:singelton} and Theorem~\ref{th:es2} imply that Conjecture~\ref{conj1} holds for graphs $G$ with $\es(G) \in \{1,2\}$. Since each color class of a proper $\chi'(G)$-edge coloring of a graph $G$ with $\es(G) = \lfloor \frac{n}{2} \rfloor$ has exactly $\lfloor \frac{n}{2} \rfloor$ edges, Conjecture~\ref{conj1} holds also for a graph $G$ with $\es(G)=\lfloor \frac{n}{2} \rfloor$, and Theorem~\ref{cor:n/2-1} confirms the conjecture when $\es(G)=\lfloor \frac{n}{2} \rfloor-1$. Thus the remaining problem is to resolve the conjecture for a (non-bipartite) graph $G$ of order $n$ with $\es(G)\in \{3,\ldots,\lfloor \frac{n}{2} \rfloor-2\}$.

In addition, we obtained characterizations of graphs $G$ with a fixed value of $\es(G)$ for some specific extremal and near-extremal values. Since not all of these characterizations are complete, we propose several problems, which would fill the missing parts of the descriptions. 

\begin{problem}
\label{pr1}
Characterize the class of connected (regular) graphs $G$ with $\es(G)=2$. 
\end{problem}

Since we did not obtain a structural characterization of the connected graphs with the chromatic edge stability index $1$, we can probably not expect a complete structural characterization of Problem~\ref{pr1}, which seems even harder. However, if we restrict to regular graphs, it might be possible to find a nice description of such graphs, similarly as for regular graphs with $\es(G)=1$ in Theorem~\ref{thm:reg1}.

The following problem arises from Section~\ref{sec:5}, addressing the Class $1$ graphs $G$ with $\es(G)=\lfloor \frac{n}{2} \rfloor-1$ for which no description was found.  

\begin{problem}
Characterize the connected Class $1$ graphs $G$ of order $n$ with $\es(G)=\lfloor \frac{n}{2} \rfloor-1$.
\end{problem}

\section*{Acknowledgments}
The research of the first author was supported by grant number (G981202) from the Sharif University of Technology. The financial support from the Slovenian Research Agency (research core funding P1-0297, and projects J1-9109, J1-1693 and J1-2452) is acknowledged by B.B. and T.D.  The sixth author acknowledges the financial support from Yazd University research affairs as Post-doc research project. The authors would like to thank the referees for their useful comments and suggestions. The authors would like to thank Mohamad Javad Sajady for his fruitful comments.

\end{document}